\documentclass[12pt]{amsart}  
\usepackage{amsmath,amssymb,amsfonts}  
\usepackage{graphicx}
\usepackage{float}

\newtheorem{theorem}[equation]{Theorem}

\newtheorem{lemma}[equation]{Lemma}

\newtheorem{remark}[equation]{Remark}

\numberwithin{equation}{section}

\title{Shafer-Fink type inequalities for arc lemniscate functions} 

\author{Minjie Wei} 
\address{\textbf{Minjie Wei}\newline
School of Science, Zhejiang Sci-Tech University, Hangzhou 310018, China}
\email{minjiewei\_zstu@163.com} 

\author{Yue He} 
\address{\textbf{Yue He}\newline
School of Science, Zhejiang Sci-Tech University, Hangzhou 310018, China}
\email{yuehe\_zstu@163.com} 

\author{Gendi Wang*}
\address{\textbf{Gendi Wang} ({\rm *Corresponding author})\newline
School of Science, Zhejiang Sci-Tech University, Hangzhou 310018, China}
\email{gendi.wang@zstu.edu.cn }

\begin{document}  

\newcounter{minutes}\setcounter{minutes}{\time}
\divide\time by 60
\newcounter{hours}\setcounter{hours}{\time}
\multiply\time by 60 \addtocounter{minutes}{-\time}
\def\thefootnote{}
\footnotetext{ {\tiny File:~\jobname.tex,
          printed: \number\year-\number\month-\number\day,
          \thehours.\ifnum\theminutes<10{0}\fi\theminutes }}
\makeatletter\def\thefootnote{\@arabic\c@footnote}\makeatother

\maketitle  

\begin{abstract}
In this paper, we investigate the monotonicity and inequalities for some functions involving the arc lemniscate and the hyperbolic arc lemniscate functions. In particular, sharp Shafer-Fink type inequalities for the arc lemniscate and the hyperbolic arc lemniscate functions are proved.
\end{abstract}

{\small \sc Keywords.} {arc lemniscate functions, hyperbolic arc lemniscate functions, lemniscate functions, hyperbolic lemniscate functions, Shafer-Fink type inequalities}

{\small \sc Mathematics Subject Classification~(2010).} {26D07, 33E05}


\section{Introduction}

The arc lemniscate sine function and the hyperbolic arc lemniscate sine function are defined as follows \cite[p.259]{bb}:
$${\rm arcsl}\,x=\int_0^x\frac{\mathrm{d} t}{\sqrt {1-t^4}},\quad\,|x|\le 1$$
and
$${\rm arcslh}\,x=\int_0^x\frac{\mathrm{d} t}{\sqrt {1+t^4}},\quad\,x\in \mathbb{R},$$
respectively.
The limiting values of the above two functions are \cite[Theorem 1.7]{bb}
\begin{equation*}
\omega={\rm arcsl}(1)=\frac 1{\sqrt 2}\mathcal{K}\left(\frac 1{\sqrt 2}\right)=\frac{\Gamma^2(1/4)}{4\sqrt{2\pi}}\approx 1.31103
\end{equation*}
and
\begin{equation*}
K={\rm arcslh}(+\infty)=\sqrt{2}\,\omega\approx 1.85407,
\end{equation*}
where
\begin{equation*}
\mathcal{K}(r)=\int_0^{\frac {\pi}2}\frac{{\rm d}\,\theta}{\sqrt{1-r^2\,\sin^2\theta}}=\int_0^1\frac{{\rm d}\,t}{\sqrt{(1-t^2)(1-r^2\,t^2)}},\quad 0<r<1
\end{equation*}
is the complete elliptic integral of the first kind.
The arc lemniscate sine function ${\rm arcsl}\,x$ shows the arc length of the lemniscate $r^2=\cos 2\theta$ from the origin to the point with radial position $x$.
The arc lemniscate sine function and the hyperbolic arc lemniscate sine function are the generalized $(2,4)$-trigonometric sine and $(2,4)$-hyperbolic sine functions \cite{t}, respectively. The generalized $(p,q)$-trigonometric and hyperbolic functions are related to the $(p,q)$-eigenvalue problem of $p$-Laplacian, which attracts many researchers' attention \cite{bv,dm,kvz,le,t}.

The arc lemniscate tangent function and the hyperbolic arc lemniscate tangent function are defined in terms of the arc lemniscate sine function and the hyperbolic arc lemniscate sine function, respectively \cite[(3.5)(3.6)]{n1}:
\begin{equation*}
\mathrm{arctl}\,x=\mathrm{arcsl}\left(\frac{x}{\sqrt[4]{1+x^4}}\right),\qquad x\in\mathbb{R}
\end{equation*}
and
$${\rm arctlh}\,x={\rm arcslh}\,\left(\frac{x}{\sqrt[4]{1-x^4}}\right),\qquad |x|<1.$$

The inverses of the above four arc lemniscate functions,
{\it the lemniscate sine function ${\rm sl}$,
the hyperbolic lemniscate sine function ${\rm slh}$,
the lemniscate tangent function ${\rm tl}$,
and the hyperbolic lemniscate tangent function ${\rm tlh}$},
have the following relations \cite[(2.11)(2.12)]{n2}:
\begin{equation}\label{rel1}
\mathrm{tl}\,x=\frac{\mathrm{sl}\,x}{\sqrt[4]{1-\mathrm{sl}^4 x}},\qquad |x|< \omega
\end{equation}
and
\begin{equation}\label{rel2}
\mathrm{tlh}\,x=\frac{\mathrm{slh}\,x}{\sqrt[4]{1+\mathrm{slh}^4 x}},\qquad |x|< K.
\end{equation}

\medskip


In 1966, Shafer proposed the following inequality \cite{s1}
\begin{equation*}
\arctan\,x > \frac{3\,x}{1+2\,\sqrt{1+x^2}},\quad\,x> 0,
\end{equation*}
which was solved next year \cite{sgmk}.
In 2011 , Chen, Cheung and Wang \cite{ccw} found the best possible numbers $b,\,c$ for the following inequalities for every $a>0$
\begin{equation*}
\frac{b\,x}{1+a\,\sqrt{1+x^2}}\le \arctan\,x \le \frac{c\,x}{1+a\,\sqrt{1+x^2}},\quad\,x\ge 0.
\end{equation*}

Fink \cite{f} found the upper bound and Mortici \cite{m1} the lower bound for the arc sine function as follows:
\begin{equation*}
\frac{3\,x}{1+ \sqrt{1-x^2}}\le \arcsin\,x \le \frac{\pi\,x}{2+\,\sqrt{1-x^2}},\quad\,0\le x\le 1.
\end{equation*}
For more refinements and extensions of such kind of inequalities for trigonometric and hyperbolic functions and other related functions, the reader is referred to \cite{dc,kvz,m,pz,qg}.

In this paper, we continue the study of the so-called Shafer-Fink type inequalities for the arc lemniscate  functions. Specifically, we try to find the best possible numbers $\alpha,\,\beta$  e.g., for the arc lemniscate sine function:
\begin{equation}\label{dingli1.3}
\frac{1}{\alpha+(1-\alpha)\sqrt{1-x^4}}<\frac{\mathrm{arcsl}\,x}{x}<\frac{1}{\beta+(1-\beta)\sqrt{1-x^4}},\quad\,0<|x|<1.
\end{equation}

\medskip

Our results are stated in the following two theorems.

\begin{theorem}\label{sf1}
The following inequalities hold:
\begin{equation}\label{arcsl4}
\frac{\omega}{1+(\omega-1)\sqrt[4]{1-x^4}}<\frac{{\rm arcsl}\,x}{x}<\frac{5}{3+2\,\sqrt[4]{1-x^4}},\quad\,0<|x|<1,
\end{equation}
\begin{equation}\label{arcslh4}
\frac{\sqrt{2}\,\omega}{(\sqrt{2}\,\omega-1)+\sqrt[4]{1+x^4}}<\frac{{\rm arcslh}\,x}{x}<\frac{5}{3+2\,\sqrt[4]{1+x^4}},\quad\,|x|>0,
\end{equation}
\begin{equation}\label{arctl4}
\frac{\omega}{(\omega-1)+\sqrt[4]{1+x^4}}<\frac{{\rm arctl}\,x}{x}<\frac{5}{2+3\,\sqrt[4]{1+x^4}},\quad\,|x|>0,
\end{equation}
\begin{equation}\label{arctlh4}
\frac{\sqrt{2}\,\omega}{1+\left(\sqrt{2}\,\omega-1\right)\sqrt[4]{1-x^4}}<\frac{{\rm arctlh}\,x}{x}<\frac{5}{2+3\,\sqrt[4]{1-x^4}},\quad\,0<|x|<1.
\end{equation}
Moreover, all the constants in the inequalities are the best possible in the sense of the form of \eqref{dingli1.3}.
\end{theorem}

\medskip

\begin{theorem}\label{sf2}
The following inequalities hold:
\begin{equation}\label{arcsl2}
\frac{5}{4+\sqrt{1-x^4}}<\frac{{\rm arcsl}\,x}{x}<\frac{\omega}{1+\left(\omega-1\right)\sqrt{1-x^4}},\quad\,0<|x|<1,
\end{equation}
\begin{equation}\label{arcslh2}
\frac{5}{4+\sqrt{1+x^4}}<\frac{{\rm arcslh}\,x}{x},\quad\,|x|>0,
\end{equation}
\begin{equation}\label{arctl2}
\frac{10}{7+3\sqrt{1+x^4}}<\frac{{\rm arctl}\,x}{x},\quad\,|x|>0,
\end{equation}
\begin{equation}\label{arctlh2}
\frac{10}{7+3\,\sqrt{1-x^4}}<\frac{{\rm arctlh}\,x}{x}<\frac{\sqrt{2}\,\omega}{1+\left(\sqrt{2}\,\omega-1\right)\sqrt{1-x^4}}\quad\,0<|x|<1.
\end{equation}
Moreover, all the constants in the inequalities are the best possible in the sense of the form of \eqref{dingli1.3}.
\end{theorem}

\medskip

\section{Basic properties}

By the definitions and the chain rule, we easily obtain the following derivative formulas of the arc lemniscate and the hyperbolic arc lemniscate functions:
\begin{equation*}
\frac{\rm d}{{\rm d}\,x}\,{\rm arcsl}\,x=(1-x^4)^{-\frac{1}{2}},\qquad |x|<1,
\end{equation*}
\begin{equation*}
\frac{\rm d}{{\rm d}\,x}\,{\rm arcslh}\,x=(1+x^4)^{-\frac{1}{2}},\qquad x\in\mathbb{R},
\end{equation*}
\begin{equation*}
\frac{\rm d}{{\rm d}\,x}\,{\rm arctl}\,x=(1+x^4)^{-\frac{3}{4}},\qquad x\in\mathbb{R},
\end{equation*}
\begin{equation*}
\frac{\rm d}{{\rm d}\,x}\,{\rm arctlh}\,x=(1-x^4)^{-\frac{3}{4}},\qquad |x|<1.
\end{equation*}

By the definitions and the inverse function theorem, we easily obtain the following derivative formulas of the lemniscate and the hyperbolic lemniscate functions:
\begin{equation*}
 \frac{\mathrm{d}}{\mathrm{d}\,x}\,{\mathrm{sl}\,x}= \sqrt{1-{\mathrm{sl}^4 x}},\qquad |x|<\omega,
 \end{equation*}
\begin{equation*}
 \frac{\mathrm{d}}{\mathrm{d}\,x}\,{\mathrm{slh}\,x}= \sqrt{1+{\mathrm{slh}^4 x}},\qquad |x|<K,
 \end{equation*}
 \begin{equation*}
 \frac{\mathrm{d}}{\mathrm{d}\,x}\,{\mathrm{tl}\,x}= (1+{\mathrm{tl}^4 x})^{\frac 34},\qquad |x|<\omega,
 \end{equation*}
\begin{equation*}
 \frac{\mathrm{d}}{\mathrm{d}\,x}\,{\mathrm{tlh}\,x}=(1-{\mathrm{tlh}^4 x})^{\frac 34},\qquad |x|<K.
 \end{equation*}

The following derivative formulas are useful:
 \begin{equation*}
 \frac{\mathrm{d}}{\mathrm{d}\,x}\,\sqrt{1-{\mathrm{sl}^4 x}}= -2{\mathrm{sl}^3 x},\qquad |x|<\omega,
 \end{equation*}
\begin{equation*}
 \frac{\mathrm{d}}{\mathrm{d}\,x}\,\sqrt{1+{\mathrm{slh}^4 x}}= 2{\mathrm{slh}^3 x},\qquad |x|<K.
 \end{equation*}

\medskip

It is proved in \cite[Lemma 4.1]{n3} that
for $0<x<1$, there hold
\begin{equation}\label{n1.5}
\mathrm{sl}\,x<x<\mathrm{tl}\,x
\end{equation}
and
\begin{equation}\label{n1.6}
\mathrm{tlh}\,x<x<\mathrm{slh}\,x. 
\end{equation}
By Lemma \ref{yinli6} in the next section, the inequalities \eqref{n1.5} and \eqref{n1.6} actually hold for $0<x<\omega$ and $0<x<K$, respectively.
Therefore, it is natural to ask:
are the functions $\mathrm{sl}\,x$ and $\mathrm{tlh}\,x$ comparable, as well as the functions $\mathrm{tl}\,x$ and $\mathrm{slh}\,x$ for $0<x<\omega$?
Since the functions $\frac{\mathrm{sl}\,x}{x}$,\,$\frac{x}{\mathrm{tl}\,x}$,\,$\frac{\mathrm{tlh}\,x}{x}$ and $\frac{x}{\mathrm{slh}\,x}$ are all less than 1 for $0<x<\omega$, we would further compare these four functions.
The following Theorem \ref{lemnfun1} shows the conclusion.

\medskip

\begin{theorem}\label{lemnfun1}
For $x\in(0,\omega)$, the following inequalities are valid:
\begin{equation}\label{dingli1.1}
{\rm tlh}\,x<{\rm sl}x<x<{\rm slh}\,x<{\rm tl}\,x,
\end{equation}
\begin{equation}\label{dingli1.2}
\frac{x}{{\rm tl}\,x}<\min\left\{\frac{{\rm tlh}x}{x},\frac{x}{{\rm slh}\,x}\right\}
\le \max\left\{\frac{ {\rm tlh}\,x}{x},\frac{x}{{\rm slh}x}\right\}
<\frac{{\rm sl}\,x}{x}.
\end{equation}
\end{theorem}

\medskip

\begin{remark}
(1) The functions $\frac{\mathrm{tlh}x}{x}$ and $\frac{x}{\mathrm{slh}x}$ are not comparable on the whole interval $(0,\omega)$ as shown in Fig. 1.\\
(2) By \cite[Theorem4.1]{n4}, we can get a similar comparison between the arc lemniscate and the hyperbolic arc lemniscate functions as \eqref{dingli1.1}:
\begin{equation}\label{n1.51}
\mathrm{arctl}\,x<\mathrm{arcslh}\,x<x<\mathrm{arcsl}\,x<\mathrm{arctlh}\,x,\quad\,0<x<1.  
\end{equation}
\end{remark}

\vspace{2mm}

\begin{figure}[H]
\centering
\includegraphics[width=8cm]{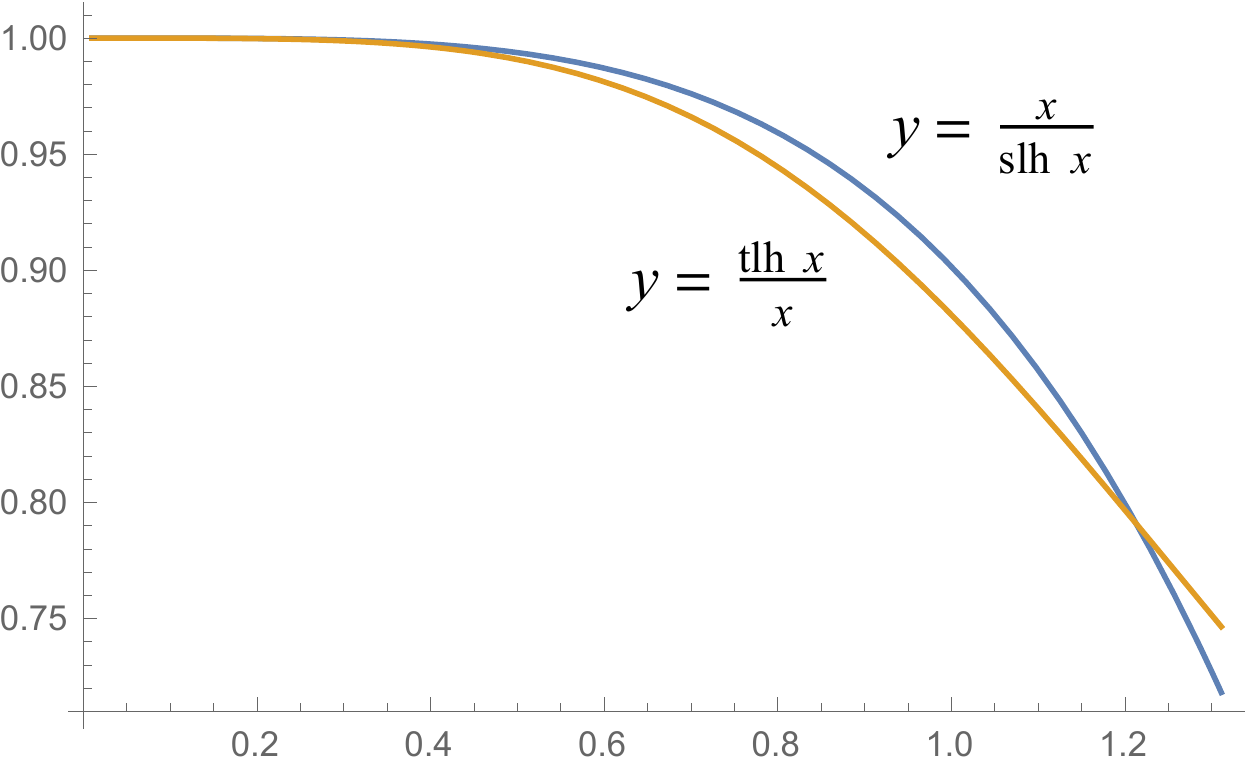}
\caption{ The functions $\frac{\mathrm{tlh}x}{x}$ and $\frac{x}{\mathrm{slh}x}$ are not comparable on the whole interval $(0,\omega)$.}
\end{figure}

To prove Theorem \ref{lemnfun1}, we need some lemmas.
The following Lemma \ref{yinli2} is of great use in deriving monotonicity properties.

\medskip

\begin{lemma}\label{yinli2}{\rm \cite[Theorem 1.25]{avv}}~({\it l'H\^opital's rule})
For $-\infty<a<b<\infty$,
let functions $f,\,g: [a,b]\rightarrow \mathbb{R}$ be continuous on $[a,b]$,
and be differentiable on $(a,b)$,
and let $g'(x)\ne 0$ on $(a,b)$.
If $f'(x)/g'(x)$ is increasing~(deceasing) on $(a,b)$, then so are
\begin{eqnarray*}
\frac{f(x)-f(a)}{g(x)-g(a)}\,\,\,\,\,\,\,and\,\,\,\,\,\,\,\,\frac{f(x)-f(b)}{g(x)-g(b)}.
\end{eqnarray*}
If $f'(x)/g'(x)$ is strictly monotone, then the monotonicity in the conclusion is also strict.
\end{lemma}

\medskip

\begin{lemma}\label{yinli31}
For $x\in (0,+\infty)$, there holds
\begin{equation}\label{yinli3.2}
{\mathrm{arctl}\,x}<{\mathrm{arcslh}\,x}.
\end{equation}
\end{lemma}
\medskip
\begin{proof}
Let $g(x)={\mathrm{arctl}\,x}-{\mathrm{arcslh}\,x}$.
By differentiation, we get
$$g'(x)=\left(1+x^4\right)^{-\frac{3}{4}}\left(1-(1+x^4)^{\frac{1}{4}} \right)<0,$$
which implies that $g$ is strictly decreasing. Hence we have $g(x)<g(0^+)=0$.
Then the inequality \eqref{yinli3.2} follows.
\end{proof}

\medskip

\begin{lemma}\label{yinli3}
For $x\in (0,\omega)$, there hold
\begin{equation}\label{yinli4.1}
{\mathrm{slh}\,x}<{\mathrm{tl}\,x},
\end{equation}
\begin{equation}\label{yinli4.2}
{\mathrm{tlh}\,x}<{\mathrm{sl}\,x},
\end{equation}
\begin{equation}\label{yinli3.1}
\frac{x}{\mathrm{slh}\,x}<\frac{\mathrm{sl}\,x}{x},
\end{equation}
\begin{equation}\label{yinli4.3}
\frac{x}{\mathrm{tl}\,x}<\frac{\mathrm{tlh}\,x}{x}.
\end{equation}
\end{lemma}
\medskip
\begin{proof}
(1)
~Given $x\in(0,\omega)$, let $y_1={\mathrm{tl}\,x}$ and $y_2={\mathrm{slh}\,x}$.
By \eqref{yinli3.2}, we obtain
\begin{equation*}
x={\mathrm{arctl}\,y_1}={\mathrm{arcslh}\,y_2}>{\mathrm{arctl}\,y_2}.
\end{equation*}
Then $y_2<y_1$ since ${\mathrm{arctl}\,x}$ is strictly increasing on $(0,+\infty)$.
Hence the inequality \eqref{yinli4.1} follows.

\medskip

(2)
~By \eqref{yinli4.1} and \eqref{rel1}, we have
\begin{equation*}
{\rm slh}^4 x<{\rm tl}^4 x=\frac{{\rm sl}^4 x}{1-{\rm sl}^4 x}
\end{equation*}
and hence
\begin{equation}\label{yinli4.4}
(1+{\rm slh}^4 x)(1-{\rm sl}^4 x)<1.
\end{equation}
Together with \eqref{rel2}, we have
\begin{equation*}
{\rm tlh}^4x=\frac{{\rm slh}^4 x}{1+{\rm slh}^4 x}<{\rm sl}^4 x,
\end{equation*}
which implies the inequality \eqref{yinli4.2}.

\medskip

(3)
~Let $f(x)=\frac{f_{11}(x)}{f_{12}(x)}$, where $f_{11}(x)={\mathrm{sl}\,x}\cdot{\mathrm{slh}\,x}$ and $f_{12}(x)=x^2$.
Then $f_{11}(0^+)=f_{12}(0^+)=0$ and
\begin{equation*}
\frac{f'_{11}(x)}{f'_{12}(x)}=\frac{{\mathrm{slh}\,x}\sqrt{1-{\mathrm{sl}^4 x}}+{\mathrm{sl}\,x} \sqrt{1+{\mathrm{slh}^4 x}}}{2\,x}\,.
\end{equation*}
Clearly, $f'_{11}(0^+)=f'_{12}(0^+)=0$.
By differentiation, we have
\begin{equation*}
\frac{f''_{11}(x)}{f''_{12}(x)}={-{\rm sl^3} x}\,{\rm slh}\,x+{\rm slh^3}\,x\,{\rm sl}\,x
+\sqrt{1-{\rm sl^4} x}\,\sqrt{1+{\rm slh^4} x}\equiv{f_{2}(x)}.
\end{equation*}
Differentiation yields
\begin{equation*}
f'_2(x)=3({\mathrm{slh}^2} x-{\mathrm{sl}^2} x)\left({\mathrm{slh}}\,x \sqrt{1-{\mathrm{sl}^4 x}}+{\mathrm{sl}\,x} \sqrt{1+{\mathrm{slh}^4 x}}\right).
\end{equation*}
By \eqref{n1.5} and \eqref{n1.6}, we obtain
$${\mathrm{sl}\,x}<x<{\mathrm{slh}\,x},\quad\,0<x<\omega,$$
which implies $f_2'(x)>0$ and hence $f_2$ is strictly increasing.
By Lemma \ref{yinli2}, we see that $f$ is strictly increasing.
Since $f(0^+)=f_{2}(0^+)=1$, we get
\begin{equation}\label{yinli3.3}
\frac{{\mathrm{sl}x}\cdot{\mathrm{slh}x}}{x^2}>1.
\end{equation}
Thus the inequality \eqref{yinli3.1} follows.

\medskip

(4)
~Let $h(x)=\frac{{\rm tl}\,x\cdot{\rm tlh}\,x}{x^2}$.
By \eqref{rel1}, \eqref{rel2}, \eqref{yinli4.4} and \eqref{yinli3.3}, we have
\begin{equation*}
 h(x)=\frac{\frac{{\rm sl}\,x\cdot{\rm slh}\,x}{x^2}}{\sqrt[4]{(1+{\rm slh}^4 x)(1-{\rm sl}^4 x)}}>1.
 \end{equation*}
This follows the inequality \eqref{yinli4.3}.
\end{proof}

\medskip

\begin{proof}[Proof of Theorem \ref{lemnfun1}]
By \eqref{n1.5},\,\eqref{n1.6},\,\eqref{yinli4.1} and \eqref{yinli4.2}, we obtain the inequalities \eqref{dingli1.1}.
Utilizing \eqref{dingli1.1},\,\eqref{yinli3.1} and \eqref{yinli4.3}, we get the inequalities \eqref{dingli1.2}.
\end{proof}

\medskip

\section{Shafer-Fink type inequalities}

In this section, we will prove the main Theorem \ref{sf1} and Theorem \ref{sf2}.
We first prove monotonicity properties of some functions involving the arc lemniscate and the hyperbolic arc lemniscate functions.

\begin{lemma}\label{yinli6}
(1) The function $f_{1}(x)\equiv\frac{\mathrm{arcsl}\,x}{x}$ is strictly increasing on $(0,1)$ with range $(1,\omega)$;\\
(2) The function $f_{2}(x)\equiv\frac{\mathrm{arcslh}\,x}{x}$ is strictly decreasing on $(0,+\infty)$ with range $(0,1)$;\\
(3) The function $f_{3}(x)\equiv\frac{\mathrm{arctl}\,x}{x}$ is strictly decreasing on $(0,+\infty)$ with range $(0,1)$;\\
(4) The function $f_{4}(x)\equiv\frac{\mathrm{arctlh}\,x}{x}$ is strictly increasing on $(0,1)$ with range $(1,\sqrt2\omega)$.\\
\end{lemma}

\begin{proof}
(1) Write $f_{1}(x)=\frac{f_{11}(x)}{f_{12}(x)}$, where $f_{11}(x)={\mathrm{arcsl}\,x}$ and $f_{12}(x)=x$.
Then $f_{11}(0^+)=f_{12}(0^+)=0$ and
\begin{equation*}
\frac{f'_{11}(x)}{f'_{12}(x)}=\frac{1}{\sqrt{1-x^4}},
\end{equation*}
which is strictly increasing.
Hence $f_{1}$ is strictly increasing by Lemma \ref{yinli2}.
The limiting value
\begin{equation*}
f_{1}(0^+)=\lim\limits_{x\to 0^{+}}\frac{f_{11}'(x)}{f_{12}'(x)}=1
\end{equation*}
and
$f_{1}(1^-)=\omega$ is clear.

\medskip

(2) Write $f_{2}(x)=\frac{f_{21}(x)}{f_{22}(x)}$, where $f_{21}(x)={\mathrm{arcslh}\,x}$ and $f_{22}(x)=x$.
Then $f_{21}(0^+)=f_{22}(0^+)=0$ and
\begin{equation*}
\frac{f'_{21}(x)}{f'_{22}(x)}=\frac{1}{\sqrt{1+x^4}},
\end{equation*}
which is strictly decreasing.
Hence $f_{2}$ is strictly decreasing by Lemma \ref{yinli2}.
The limiting value
\begin{equation*}
f_{2}(0^+)=\lim\limits_{x\to 0^{+}}\frac{f_{21}'(x)}{f_{22}'(x)}=1
\end{equation*}
and
$f_{2}(+\infty)=0$ is clear.

\medskip

(3) Write $f_{3}(x)=\frac{f_{31}(x)}{f_{32}(x)}$, where $f_{31}(x)={\mathrm{arctl}\,x}$ and $f_{32}(x)=x$.
Then $f_{31}(0^+)=f_{32}(0^+)=0$ and
\begin{equation*}
\frac{f'_{31}(x)}{f'_{32}(x)}=(1+x^4)^{-\frac{3}{4}},
\end{equation*}
which is strictly decreasing.
Hence $f_{3}$ is strictly decreasing by Lemma \ref{yinli2}.
The limiting value
\begin{equation*}
f_{3}(0^+)=\lim\limits_{x\to 0^{+}}\frac{f_{31}'(x)}{f_{32}'(x)}=1
\end{equation*}
and
$f_{3}(+\infty)=0$ is clear.

\medskip

(4) Write $f_{4}(x)=\frac{f_{41}(x)}{f_{42}(x)}$, where $f_{41}(x)={\mathrm{arctlh}\,x}$ and $f_{42}(x)=x$.
Then $f_{41}(0^+)=f_{42}(0^+)=0$ and
\begin{equation*}
\frac{f'_{41}(x)}{f'_{42}(x)}=(1-x^4)^{-\frac{3}{4}},
\end{equation*}
which is strictly increasing.
Hence $f_{4}$ is strictly increasing by Lemma \ref{yinli2}. The limiting value
\begin{equation*}
f_{4}(0^+)=\lim\limits_{x\to 0^{+}}\frac{f_{41}'(x)}{f_{42}'(x)}=1
\end{equation*}
and
$f_{4}(1^-)=\sqrt2\omega$ is clear.
\end{proof}

\medskip

\begin{lemma}\label{yinli7}
(1) The function $g_{1}(x)\equiv\frac{x-\sqrt[4]{1-x^4}\,\mathrm{arcsl}\,x}{\mathrm{arcsl}\,x-x}$ is strictly increasing on $(0,1)$ with range $(\frac{3}{2},\frac{1}{\omega-1})$;\\
(2) The function $g_{2}(x)\equiv\frac{\sqrt[4]{1+x^4}\,\mathrm{arcslh}\,x-x}{x-\mathrm{arcslh}\,x}$ is strictly decreasing on $(0,+\infty)$ with range $(\sqrt{2}\,\omega-1,\frac{3}{2})$;\\
(3) The function $g_{3}(x)\equiv\frac{\sqrt[4]{1+x^4}\,\mathrm{arctl}\,x-x}{x-\mathrm{arctl}\,x}$ is strictly decreasing on $(0,+\infty)$ with range $(\omega-1,\frac{2}{3})$;\\
(4) The function $g_{4}(x)\equiv\frac{x-\sqrt[4]{1-x^4}\,\mathrm{arctlh}\,x}{\mathrm{arctlh}\,x-x}$ is strictly increasing on $(0,1)$ with range $(\frac{2}{3},\frac{1}{\sqrt{2}\omega-1})$.
\end{lemma}
\medskip
\begin{proof}
(1) Write $g_{1}(x)=\frac{g_{11}(x)}{g_{12}(x)}$,
where $g_{11}(x)=x-\sqrt[4]{1-x^4}\,\mathrm{arcsl}\,x$ and $g_{12}(x)=\mathrm{arcsl}\,x-x$.
Then $g_{11}(0^+)=g_{12}(0^+)=0$ and
\begin{equation*}
\frac{g'_{11}(x)}{g'_{12}(x)}=\frac{x^3\,(1-x^4)^{-\frac{3}{4}}\,{\rm arcsl}\,x-(1-x^4)^{-\frac{1}{4}}+1}{(1-x^4)^{-\frac{1}{2}}-1}.
\end{equation*}
Clearly, $g'_{11}(0^+)=g'_{12}(0^+)=0$.
By differentiation, we get
\begin{equation*}
\frac{g''_{11}(x)}{g''_{12}(x)}=\frac{3}{2}\,\frac{\mathrm{arcsl}\,x}{x}\,\frac{1}{\sqrt[4]{1-x^4}},
\end{equation*}
which is strictly increasing by Lemma \ref{yinli6}(1).
Hence $g_{1}$ is strictly increasing by Lemma \ref{yinli2}.
The limiting value
\begin{equation*}
g_{1}(0^+)=\lim\limits_{x\to 0^{+}}{\frac{g''_{11}(x)}{g''_{12}(x)}}=\frac{3}{2}
\end{equation*}
and
$g_{1}(1^-)=\frac{1}{\omega-1}$ is clear.

\medskip

(2) Write $g_{2}(x)=\frac{g_{21}(x)}{g_{22}(x)}$, where $g_{21}(x)=\sqrt[4]{1+x^4}\,\mathrm{arcslh}\,x-x$ and $g_{22}(x)=x-\mathrm{arcslh}\,x$.
Then $g_{21}(0^+)=g_{22}(0^+)=0$ and
\begin{equation*}
\frac{g_{21}'(x)}{g_{22}'(x)}=\frac{x^3\,(1+x^4)^{-\frac{3}{4}}\,{\rm arcslh}\,x+(1+x^4)^{-\frac{1}{4}}-1}{1-{(1+x^4)}^{-\frac{1}{2}}}. 
\end{equation*}
Clearly, $g'_{21}(0^+)=g'_{22}(0^+)=0$.
By differentiation, we get
\begin{equation*}
\frac{g''_{21}(x)}{g''_{22}(x)}=\frac{3}{2}\,\frac{\mathrm{arcslh}\,x}{x}\,\frac{1}{\sqrt[4]{1+x^4}},
\end{equation*}
which is strictly decreasing by Lemma \ref{yinli6}(2).
Hence $g_{2}$ is strictly decreasing by Lemma \ref{yinli2}.
The limiting value
\begin{equation*}
g_{2}(0^+)=\lim\limits_{x\to 0^{+}}{\frac{g''_{21}(x)}{g''_{22}(x)}}=\frac{3}{2}
\end{equation*}
and
$g_{2}(+\infty)=\sqrt{2}\,\omega-1$ is clear.

\medskip

(3) Write $g_{3}(x)=\frac{g_{31}(x)}{g_{32}(x)}$, where $g_{31}(x)=\sqrt[4]{1+x^4}\,\mathrm{arctl}\,x-x$ and $g_{32}(x)=x-\mathrm{arctl}\,x$.
Then $g_{31}(0^+)=g_{32}(0^+)=0$ and
\begin{equation*}
\frac{g'_{31}(x)}{g'_{32}(x)}=\frac{x^3\,(1+x^4)^{-\frac{3}{4}}\,{\rm arctl}\,x+(1+x^4)^{-\frac{1}{2}}-1}{1-{(1+x^4)}^{-\frac{3}{4}}}\,.
\end{equation*}
Clearly, $g'_{31}(0^+)=g'_{32}(0^+)=0$.
By differentiation, we get
\begin{equation*}
\frac{g''_{31}(x)}{g''_{32}(x)}=\frac{\mathrm{arctl}\,x}{x}-\frac{1}{3}\sqrt[4]{1+x^4},
\end{equation*}
which is strictly decreasing by Lemma \ref{yinli6}(3).
Hence $g_{3}$ is strictly decreasing by Lemma \ref{yinli2}.
The limiting value
\begin{equation*}
g_{3}(0^+)=\lim\limits_{x\to 0^{+}}{\frac{g''_{31}(x)}{g''_{32}(x)}}=\frac{2}{3}
\end{equation*}
and
$g_{3}(+\infty)=\omega-1$ is clear.

\medskip

(4) Write $g_{4}(x)=\frac{g_{41}(x)}{g_{42}(x)}$,
where $g_{41}(x)=x-\sqrt[4]{1-x^4}\,\mathrm{arctlh}\,x$ and $g_{42}(x)=\mathrm{arctlh}\,x-x$.
Then $g_{41}(0^+)=g_{42}(0^+)=0$ and
\begin{equation*}
\frac{g'_{41}(x)}{g'_{42}(x)}=\frac{x^3\,(1-x^4)^{-\frac{3}{4}}\,{\rm arctlh}\,x-(1-x^4)^{-\frac{1}{2}}+1}{{(1-x^4)}^{-\frac{3}{4}}-1}\,.
\end{equation*}
Clearly, $g'_{41}(0^+)=g'_{42}(0^+)=0$.
By differentiation, we get
\begin{equation*}
\frac{g''_{41}(x)}{g''_{42}(x)}=\frac{\mathrm{arctlh}\,x}{x}-\frac{1}{3}\sqrt[4]{1-x^4}\,,
\end{equation*}
which is strictly increasing by Lemma \ref{yinli6}(4).
Hence $g_{4}$ is strictly increasing by Lemma \ref{yinli2}.
The limiting value
\begin{equation*}
g_{4}(0^+)=\lim\limits_{x\to 0^{+}}{\frac{g''_{41}(x)}{g''_{42}(x)}}=\frac{2}{3}
\end{equation*}
and
$g_{4}(1^-)=\frac{1}{\sqrt{2}\omega-1}$ is clear.
\end{proof}

\medskip

\begin{proof}[Proof of Theorem \ref{sf1}]
The inequalities \eqref{arcsl4}~--~\eqref{arctlh4} follow from the odevity of the arc lemniscate and the hyperbolic arc lemniscate functions and the monotonicity properties of the functions in Lemma \ref{yinli7}. It is easy to see that the constants in the inequalities are best possible from the ranges and the monotonicity of the corresponding functions in Lemma \ref{yinli7}.
\end{proof}

\medskip

\begin{lemma}\label{yinli51}
(1) The function $f_{1}(x)\equiv\frac{\sqrt{1-x^4}\,\mathrm{arcsl}\,x}{x}$ is strictly decreasing on $(0,1)$ with range $(0,1)$;\\
(2) The function $f_{2}(x)\equiv\frac{\sqrt{1+x^4}\,\mathrm{arcslh}\,x}{x}$ is strictly increasing on $(0,+\infty)$ with range $(1,+\infty)$;\\
(3) The function $f_{3}(x)\equiv\frac{\sqrt[4]{1+x^4}\,\mathrm{arctl}\,x}{x}$ is strictly increasing on $(0,+\infty)$ with range $(1,\omega)$;\\
(4) The function $f_{4}(x)\equiv\frac{\sqrt[4]{1-x^4}\,\mathrm{arctlh}\,x}{x}$ is strictly decreasing on $(0,1)$ with range $(0,1)$.
\end{lemma}
\medskip
\begin{proof}
(1) Write $f_{1}(x)=\frac{f_{11}(x)}{f_{12}(x)}$, where $f_{11}(x)=\sqrt{1-x^4}\,{\rm arcsl}\,x$ and $f_{12}(x)=x$.
Then $f_{11}(0^+)=f_{12}(0^+)=0$ and
\begin{equation*}
\frac{f'_{11}(x)}{f'_{12}(x)}=1-\frac{2x^3\,{\rm arcsl}\,x}{\sqrt{1-x^4}},
\end{equation*}
which is strictly decreasing.
Hence $f_{1}$ is strictly decreasing by Lemma \ref{yinli2}.
The limiting value
\begin{equation*}
f_{1}(0^+)=\lim\limits_{x\to 0^{+}}\frac{f_{11}'(x)}{f_{12}'(x)}=1
\end{equation*}
and
$f_{1}(1^-)=0$ is clear.

\medskip

(2) Write $f_{2}(x)=\frac{f_{21}(x)}{f_{22}(x)}$, where $f_{21}(x)=\sqrt{1+x^4}\,{\rm arcslh}\,x$ and $f_{22}(x)=x$.
Then $f_{21}(0^+)=f_{22}(0^+)=0$ and
\begin{equation*}
\frac{f'_{21}(x)}{f'_{22}(x)}=1+\frac{2x^3\,{\rm arcslh}\,x}{\sqrt{1+x^4}}
=1+\frac{2x\,{\rm arcslh}\,x}{\sqrt{1+\frac{1}{x^4}}}\,,
\end{equation*}
which is strictly increasing.
Hence $f_{2}$ is strictly increasing by Lemma \ref{yinli2}.
The limiting value
\begin{equation*}
f_{2}(0^+)=\lim\limits_{x\to 0^{+}}\frac{f_{21}'(x)}{f_{22}'(x)}=1
\end{equation*}
and
$f_{2}(+\infty)=+\infty$ is clear.

\medskip

(3) By differentiation, we have
\begin{equation*}
f'_{3}(x)=\frac{(1+x^4)^{-\frac{3}{4}}\left((1+x^4)^{\frac{1}{4}}-\frac{\mathrm{arctl}\,x}{x}\right)}{x}.
\end{equation*}
By Lemma \ref{yinli6}(3), we get
\begin{equation*}
(1+x^4)^{\frac{1}{4}}>1>\frac{\mathrm{arctl}\,x}{x}.
\end{equation*}
Then $f'_{3}(x)>0$ and hence $f_{3}$ is strictly increasing.
The limiting value $f_{3}(0^+)=1$ follows from Lemma \ref{yinli6}(3)
and $f_{3}(+\infty)=\omega$ is clear.

\medskip

(4) By differentiation, we have
\begin{equation*}
f'_{4}(x)=\frac{(1-x^4)^{-\frac{3}{4}}\left((1-x^4)^{\frac{1}{4}}-\frac{\mathrm{arctlh}\,x}{x}\right)}{x}.
\end{equation*}
By Lemma \ref{yinli6}(4), we get
\begin{equation*}
(1-x^4)^{\frac{1}{4}}<1<\frac{\mathrm{arctlh}\,x}{x}.
\end{equation*}
Then $f'_{4}(x)<0$ and hence $f_{4}$ is strictly decreasing.
The limiting value $f_{4}(0^+)=1$ follows from Lemma \ref{yinli6}(4)
and $f_{4}(1^-)=0$ is clear.
\end{proof}

\medskip

\begin{lemma}\label{yinli5}
(1) The function $h_{1}(x)\equiv\frac{x-\sqrt{1-x^4}\,\mathrm{arcsl}\,x}{\mathrm{arcsl}\,x-x}$ is strictly decreasing on $(0,1)$
with range $(\frac{1}{\omega-1},4)$;\\
(2) The function $h_{2}(x)\equiv\frac{\sqrt{1+x^4}\,\mathrm{arcslh}\,x-x}{x-\mathrm{arcslh}\,x}$ is strictly increasing on $(0,+\infty)$
with range $(4,+\infty)$.\\
(3) The function $h_{3}(x)\equiv\frac{\sqrt{1+x^4}\,\mathrm{arctl}\,x-x}{x-\mathrm{arctl}\,x}$ is strictly increasing on $(0,+\infty)$
with range $(\frac 73,+\infty)$.\\
(4) The function $h_{4}(x)\equiv\frac{x-\sqrt{1-x^4}\,\mathrm{arctlh}\,x}{\mathrm{arctlh}\,x-x}$~ is strictly decreasing on $(0,1)$
with range $(\frac{1}{\sqrt 2 \omega-1},\frac 73)$.
\end{lemma}
\medskip
\begin{proof}
(1) Write $h_{1}(x)=\frac{h_{11}(x)}{h_{12}(x)}$, where $h_{11}(x)=x-\sqrt{1-x^4}\,\mathrm{arcsl}\,x$ and $h_{12}(x)=\mathrm{arcsl}\,x-x$.
Then $h_{11}(0^+)=h_{12}(0^+)=0$ and
\begin{equation*}
\frac{h'_{11}(x)}{h'_{12}(x)}=\frac{h_{13}(x)}{h_{14}(x)},
\end{equation*}
where $h_{13}(x)=2\,x^3\,{\rm arcsl}\,x$ and $h_{14}(x)=1-\sqrt{1-x^4}$.
Clearly, $h_{13}(0^+)=h_{14}(0^+)=0$.
By differentiation, we get
\begin{equation*}
\frac{h_{13}'(x)}{h_{14}'(x)}=3f_1(x)+1,
\end{equation*}
where $f_1(x)$ is the same as in Lemma \ref{yinli51}(1).
Hence $h_{1}$ is strictly decreasing by Lemma \ref{yinli51}(1) and Lemma \ref{yinli2}.
The limiting value
\begin{equation*}
h_{1}(0^+)=\lim\limits_{x\to 0^{+}}{\left(3f_1(x)+1\right)}=4
\end{equation*}
and
$h_{1}(1^-)=\frac{1}{\omega-1}$ is clear.

\medskip

(2) Write $h_{2}(x)=\frac{h_{21}(x)}{h_{22}(x)}$, where $h_{21}(x)=\sqrt{1+x^4}\,\mathrm{arcslh}\,x-x$ and $h_{22}(x)=x-\mathrm{arcslh}\,x$.
Then $h_{21}(0^+)=h_{22}(0^+)=0$ and
\begin{equation*}
\frac{h'_{21}(x)}{h'_{22}(x)}=\frac{h_{23}(x)}{h_{24}(x)},
\end{equation*}
where $h_{23}(x)=2\,x^3\,{\rm arcslh}\,x$ and $h_{24}(x)=\sqrt{1+x^4}-1$.
Clearly, $h_{23}(0^+)=h_{24}(0^+)=0$.
By differentiation, we get
\begin{equation*}
\frac{h_{23}'(x)}{h_{24}'(x)}=3f_2(x)+1,
\end{equation*}
where $f_2(x)$ is the same as in Lemma \ref{yinli51}(2).
Hence $h_{2}$ is strictly increasing by Lemma \ref{yinli51}(2) and Lemma \ref{yinli2}.
The limiting value
\begin{equation*}
h_{2}(0^+)=\lim\limits_{x\to 0^{+}}{\left(3f_2(x)+1\right)}=4
\end{equation*}
and
$h_{2}(+\infty)=+\infty$ is clear.

\medskip

(3) Write $h_{3}(x)=\frac{h_{31}(x)}{h_{32}(x)}$,  where $h_{31}(x)=\sqrt{1+x^4}\,\mathrm{arctl}\,x-x$ and $h_{32}(x)=x-\mathrm{arctl}\,x$.
Then $h_{31}(0^+)=h_{32}(0^+)=0$ and
\begin{equation*}
\frac{h_{31}'(x)}{h_{32}'(x)}=\frac{2\,x^3\,(1+x^4)^{-\frac{1}{2}}\,{\rm arctl}\,x+(1+x^4)^{-\frac{1}{4}}-1}{1-{(1+x^4)}^{-\frac{3}{4}}}\,.
\end{equation*}
Clearly, $h_{31}'(0^+)=h_{32}'(0^+)=0$.
By differentiation, we get
\begin{equation*}
\frac{h_{31}''(x)}{h_{32}''(x)}=\frac{2}{3}(3+x^4)f_3(x)+\frac{1}{3}\sqrt{1+x^4},
\end{equation*}
where $f_3(x)$ is the same as in Lemma \ref{yinli51}(3).
Hence $h_{3}$ is strictly increasing by Lemma \ref{yinli51}(3) and Lemma \ref{yinli2}.
The limiting value
\begin{equation*}
h_{3}(0^+)=\lim\limits_{x\to 0^{+}}{\frac{h_{31}''(x)}{h_{32}''(x)}}=\frac{7}{3}
\end{equation*}
and
$h_{3}(+\infty)=+\infty$ is clear.

\medskip

(4) Write $h_{4}(x)=\frac{h_{41}(x)}{h_{42}(x)}$, where $h_{41}(x)=x-\sqrt{1-x^4}\,\mathrm{arctlh}\,x$ and $h_{42}(x)=\mathrm{arctlh}\,x-x$.
Then $h_{41}(0^+)=h_{42}(0^+)=0$ and
\begin{equation*}
\frac{h'_{41}(x)}{h'_{42}(x)}=\frac{2\,x^3\,(1-x^4)^{-\frac{1}{2}}\,{\rm arctlh}\,x-(1-x^4)^{-\frac{1}{4}}+1}{{(1-x^4)}^{-\frac{3}{4}}-1}\,.
\end{equation*}
Clearly, $h'_{41}(0^+)=h'_{42}(0^+)=0$.
By differentiation, we get
\begin{equation*}
\frac{h''_{41}(x)}{h''_{42}(x)}=\frac{2}{3}\,(3-x^4)\,f_4(x)+\frac{1}{3}\,\sqrt{1-x^4}\,,
\end{equation*}
where $f_4(x)$ is the same as in Lemma \ref{yinli51}(4).
Hence $h_{4}(x)$ is strictly decreasing by Lemma \ref{yinli51}(4) and Lemma \ref{yinli2}.
The limiting value
\begin{equation*}
h_{4}(0^+)=\lim\limits_{x\to 0^{+}}{\frac{h_{41}''(x)}{h_{42}''(x)}}=\frac{7}{3}
\end{equation*}
and
$h_{4}(1^-)=\frac{1}{\sqrt 2 \omega-1}$ is clear.
\end{proof}

\medskip

\begin{proof}[Proof of Theorem \ref{sf2}]
The inequalities \eqref{arcsl2}~--~\eqref{arctlh2} follow from Lemma \ref{yinli5} with a similar argument in the proof of Theorem \ref{sf1} .
\end{proof}

\medskip

\begin{remark}
In the recent paper \cite{dc}, the authors considered the following problem:
to decide the best possible constants $a_1$ and $b_1$ such that
the inequalities
$$\frac{a_1}{4+\sqrt{1-x^4}}<\frac{{\rm arcsl}\,x}{x}<\frac{b_1}{4+\sqrt{1-x^4}}$$
hold for $0<|x|<1$. Similar problems for several other arc lemniscate functions were also considered in the same paper.
Since the constants in the denominators are fixed, these problems are not the same as ours in this paper.
Our results in Theorem \ref{sf1} and Theorem \ref{sf2} refine the related inequalities in \cite{dc}.
\end{remark}

\subsection*{Acknowledgments}
 This research was supported by National Natural Science Foundation of China (NNSFC) under Grant No.11601485 and No.11771400,
 and Science Foundation of Zhejiang Sci-Tech University (ZSTU) under Grant No.16062023\,-Y.




\end{document}